\documentclass[10pt,oneside]{amsart}
\usepackage{amssymb, amscd, amsmath, amsthm, latexsym, hyperref,ams refs}
\usepackage{pb-diagram, fancyhdr, graphicx, psfrag, enumerate}
\usepackage[text={6.3in,8.5in},centering,letterpaper,dvips]{geometry}
\newcommand{\compactlist}{\begin{list}{$\bullet$}{\setlength{\leftmargin}{1em}}}

\def\Z{{\mathbb Z}}
\def\F{{\mathbb F}}

\def\Q{{\mathbb Q}}

\def\co{\colon\!}

\def\calp{\mathcal{P}}

\def\calc{\mathcal{C}}
\def\calm{\mathcal{M}}
\def\calg{\mathcal{G}}

\def\calt{\mathcal{T}}
\def\cals{\mathcal{S}}
\def\calq{\mathcal{Q}}
\def\cs{\mathbin{\#}}
\DeclareMathOperator{\cfk}{\rm CFK}
\newcommand{\spinc}{\ifmmode{{\mathfrak s}}\else{${\mathfrak s}$\ }\fi}
\newcommand{\spinct}{\ifmmode{{\mathfrak t}}\else{${\mathfrak t}$\ }\fi}

\newcommand{\fig}[2] { \includegraphics[scale=#1]{#2} }

\DeclareMathOperator*{\Moplus}{\text{\raisebox{0.25ex}{\scalebox{0.7}{$\bigoplus$}}}}

\newtheorem{theorem}{Theorem}
\newtheorem{lemma}[theorem]{Lemma}
\newtheorem{corollary}[theorem]{Corollary}
\newtheorem{conjecture}[theorem]{Conjecture}

\theoremstyle{definition}

\begin{document}
\title{Primary decompositions of knot concordance}
\author{Charles Livingston}
\thanks{This work was supported by a grant from the National Science Foundation, NSF-DMS-1505586.}
\address{Charles Livingston: Department of Mathematics, Indiana University, Bloomington, IN 47405 }
\email{livingst@indiana.edu}


\begin{abstract} Jerome Levine defined for all $n>0$ a homomorphism from the smooth concordance group of knots in dimension $2n+1$ to an algebraically defined group $\calg^\Q$.  This algebraic concordance group splits as a direct sum of groups indexed by polynomials.  For $n>1$ the homomorphism is injective.  This leads  to what is called a primary decomposition theorem.  In the classical dimension, the   kernel of this homomorphism includes the smooth concordance group of topologically slice knots, and Jae Choon Cha has begun studying possible primary decompositions of this subgroup.  Here we will show that   primary decompositions of a strong type cannot exist.

In more detail, it is shown that there exists  a topologically slice  knot $K$ for which there is a factorization of its  Alexander polynomial, $\Delta_K(t)  = f_1(t)f_2(t)$, where $f_1$ and $f_2$ are relatively prime and each is the Alexander polynomial of a topologically slice knot, but $K$ is not smoothly concordant to any connected sum $K_1 \cs K_2$  for which $\Delta_{K_i}(t) =  f_i(t)^{n_i}$ for any nonnegative integers $n_i$.      
\end{abstract}

\maketitle


\section{Introduction}

A central problem in  three-dimensional  knot concordance theory in the smooth category consists of understanding  $\calt$, the concordance group of topologically slice knots.  Freedman~\cites{MR1201584,MR679066} proved that the subgroup $\calt^1$ generated by knots with Alexander polynomial one satisfies $\calt^1 \subset \calt$.  Early work proving that $\calt^1$ is nontrivial includes that of Akbulut, Casson, and Cochran-Gompf~\cite{MR976591}; in~\cite{MR1334309}  it was shown that $\calt^1$ contains an infinitely generated free subgroup.  Using  Heegaard Floer theory, in~\cite{MR3466802, MR2955197} it was shown that $\calt / \calt^1$ contains an infinitely generated free subgroup and infinite two-torsion.  

Recently, Jae Choon Cha~\cite{arxiv:1910.14629}  has undertaken an in-depth investigation of  {\it primary decompositions} of $\calt$.   
One motivation for studying primary decompositions arises from Levine's work~\cite{MR0253348} in which it was shown that for all integers $n>0$ there is a homomorphism from the smooth concordance group of knotted $2n-1$ spheres in $S^{2n+1}$ to a group called the  {\it rational algebraic concordance group}:    $\psi_{2n-1} \co \calc_{2n-1} \to \calg^\Q$.  It was also proved that there is a decomposition $\calg^\Q \cong \oplus_{p\in \mathcal{A}}  \calg^\Q_p$, where $\mathcal{A}$ is the set of irreducible Alexander polynomials.  Such a decomposition does not exist  using integer rather than  rational coefficients, but the failure was completely analyzed by Stoltzfus~\cite{MR0467764}.  In all odd higher dimensions $\psi_{2n-1}$ is injective, leading to decomposition theorems for knot concordance groups.  In the classical dimension, $n=1$, the map $\psi_1$ is not injective~\cite{MR900252}; the kernel is infinitely generated, contained the subgroup $\calt$ of topologically slice knots.

To briefly summarize the perspective of Cha's work,  we  let  $\calq \subset \Z[t]$ be the set of irreducible  polynomials  $q(t)$ satisfying $q(1) = 1$.    Let 
$$\calp = \{ q(t)q(t^{-1})  \in \Z[t,t^{-1}] \ \big| \ q(t) \in \calq\}.$$  According to Fox and Milnor~\cite{MR0211392}, if a knot $K$ is smoothly  slice, then its Alexander polynomial is a product of elements in $\calp$.  The same result holds for topologically locally flat slice knots, as proved   using the existence of normal bundles for locally flat disks (see Freedman-Quinn~\cite[Section 9.3]{MR1201584}).
According to Terasaka~\cite{MR0117736}, every product of elements in $\calp$ is the Alexander polynomial of some slice knot.

Given any subset $\calp_0 \subset \calp$,  let $\calt^{\calp_0} \subset \calt$ denote the subgroup generated by topologically slice knots with Alexander polynomial a product of polynomials  $p$ for  $p \in \calp_0$.  In the case that $\calp_0$ is a singleton  $\{p\}$, we write $\calt^p$.  Hence, as above, $\calt^1$ denotes the subgroup generated by knots with Alexander polynomial one.    Notice that for any pair of elements   $p, q  \in \calp$, we have $\calt^1 \subset \calt^p \cap \calt^q$.  Thus, in the following conjecture  it is necessary to consider the quotients $\calt^p_\Delta = \calt^p / \calt^1$ and $\calt_\Delta = \calt / \calt^1$.  

\begin{conjecture} The canonical homomorphisms    $\calt^p_\Delta \to \calt_\Delta$  induce an isomorphism  $$\Phi \co \Moplus_{p \in \calp}\calt^p_\Delta \to \calt_\Delta.$$ 
\end{conjecture}

In~\cite{arxiv:1910.14629}, Cha identifies and studies a specific  infinite set $\calp_0 \subset \calp$ with two properties:  first, for all $p \in \calp_0$, he proves that   $\calt^p_\Delta$ contains an infinitely generated free subgroup $\cals^p_\Delta$; second, he proves that the  restriction of $\Phi$ is injective on $\oplus_{p \in \calp^0} \cals^p_\Delta$.  

The main goal of this paper is to provide a counterexample to a    splitting property related to the surjectivity of $\Phi$,  considered by Cha under the name {\it  strong existence} (see \cite[Appendix A]{arxiv:1910.14629}). 
Although this does not provide a counterexample to the conjecture, it adds to the evidence  that the conjecture is false.  More specifically, it indicates  that $\Phi$ is probably not surjective.

\begin{theorem}\label{thm:main} There exists a set of three polynomials, $\calp_0 = \{f_1, f_2,f_3\} \subset \calp$, such that the natural homomorphism  $$\calt^{f_1}_{\Delta}\oplus \calt^{f_2}_{\Delta} \oplus \calt^{f_3}_{\Delta} \to \calt^{\calp_0}_\Delta$$ is not surjective.
\end{theorem}

The use of three factors is an artifact of the proof.  It will be clear that without the restriction of irreducibility for elements in $\calq$, we could have used two factors, as was stated in the abstract.    To be more precise, there is the following statement.

\begin{theorem}  There exist  Alexander polynomials $f_1(t)$ and $f_2(t)$ having no common factors and a topologically slice knot $K$ with $\Delta_K(t) = f_1(t)^2f_2(t)^2$ such that $K$ is not concordant to any connected sum of knots $K_1 \cs K_2$ where $\Delta_{K_i}(t) = f_i(t)^{n_i}$ and $n_1, n_2 \in \Z$.
\end{theorem} 
Notice that it follows that $K$ is a topologically slice knot that is not smoothly concordant to a knot with Alexander polynomial one.  The first  examples of  such knots were described in~\cite{MR2955197}.  The example and proof here are closely related to that earlier work. 

\vskip.05in
\noindent{\it Acknowledgments} Thanks are due to Jae Choon Cha for sharing with me    drafts of his ongoing work on primary decompositions and for his repeatedly offering thoughtful commentary on this note.  Discussions with Se-Goo Kim and Taehee Kim were also very helpful.

\section{Rational homology cobordism and the construction of an example}

The proof of Theorem~\ref{thm:main} can be  reduced to a result concerning rational homology cobordism, as we now describe.  Recall first that the Alexander polynomial of a knot determines the order of the first  homology of $M(K)$,  the 2--fold cyclic branched cover of $S^3$ branched over $K$: $\big| H_1(M(K))\big| = \big|\Delta_K(-1)\big|$. Also, 2--fold branched covers of concordant knots are rationally homology cobordant. 

Figure~\ref{fig:knot} is a schematic diagram of a  topologically slice knot $K_n(J_n)$ (in the case of $n = 3$); we will be restricting to the case of  $n \equiv 3 \mod 4$.  These knots  bound       punctured Klein bottles in $S^3$.  In each, the left band  has framing 0,  there are   $n$ half-twists between the bands,  and the knot $J_n $ is the connected sum of $(3n-1)/4$ copies of the positive clasped untwisted double of the trefoil knot, $D(T(2,3))$.    Similar knots were used in~\cite{MR2955197} to prove the nontriviality of $\calt_\Delta = \calt /\calt^1$.  Henceforth, will write $K_n$ for $K_{n}(J_{n})$.   A quick calculation in Section~\ref{sec:alex} will show   that if $n=pq$ for odd primes $p$ and $q$, then  $$\Delta_{K_n}(t) = \big(    \phi_{2p}(t) \phi_{2q}(t)\phi_{2pq}(t)  \big)^2,$$ where $\phi_k(t)$ denotes the $k$--cyclotomic polynomial.    
\begin{figure}
\fig{.6}{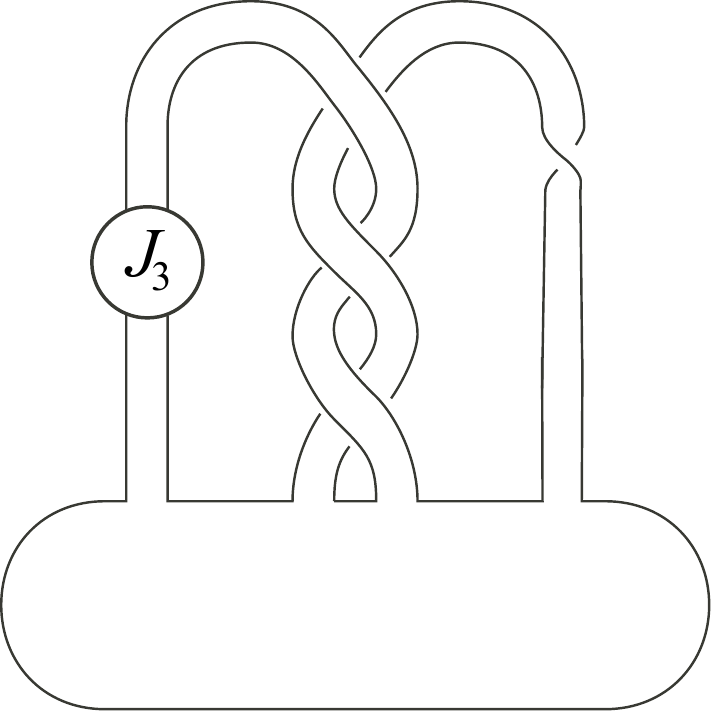}
\caption{Knot}\label{fig:knot}
\end{figure}

\begin{theorem} If $n = pq$, then the homology of the 2--fold branched cover of $K_n$ satisfies $\big| H_1(M(K_n))\big| = n^2$.  If $K_n$ is concordant to a connected sum $L_1 \cs L_2 \cs L_3$, with $\Delta_{L_1}(t) =  \phi_{2p}(t)^{m_1}$,    $\Delta_{L_2}(t) =  \phi_{2q}(t)^{m_2}$, and $\Delta_{L_3}(t) = \phi_{2pq}(t)^{m_3}$, then $M(K_n) $ is rationally homology cobordant to a connected sum $M_1 \cs M_2 \cs M_3$, where $\big|H_1(M_1)\big| = p^{m_1}$,  $\big|H_1(M_2) \big| = q^{m_2}$ and  $\big|H_1(M_3) \big| = 1$. 

\end{theorem}

\begin{proof}  These all follow  immediately from the facts that $  \phi_{2p}(-1) = p$,    $  \phi_{2q}(-1)  =q $, and $ \phi_{2pq}(-1) = 1$.  See Lemma~\ref{lemma:cyclotomic} for details.

\end{proof}  

The topic of~\cite{MR3259765} was the general problem of finding a primary splitting  of the rational homology cobordism group.  We will show that the techniques used there can be applied to prove that for $n=15$, a rational homology  cobordism from $M(K_n)$ to such a connected sum, $M_1 \cs M_2 \cs M_3$, does not exist. Notice that if it did exist,  we could let $N_2 =M_2 \cs M_3$, and reduce  our  work  to obstructing the existence of a rational homology cobordism to a connected sum of two manifolds, $M_1 \cs N_2$.  Our goal will be to prove that $M(K_n)$ is not rational homology cobordant to any connected sum $M_1 \cs M_2$, where $\big|H_1(M_1)\big| = 3^{m_{1}}$ and   $\big|H_1(M_2)\big| = 5^{m_{2}}$ for some integers $m_1$ and $m_2$.

\section{Obstructions from $d$--invariants}

We use the following notation:  for any abelian group $G$ and prime integer $p$,   let $G_{(p)}$ denote the subgroup consisting of all elements of order $p^n$ for some $n$.

All the three-manifolds $M$ we will be   working with are $\Z_2$--homology three-spheres.  Since $H_1(M)$ is of odd order, there are natural identifications: $H_1(M)  \cong H^2(M) \cong \text{Spin}^c(M)$.  For $g \in H_1(M)$ we denote by $d(K, g)$ the Heegaard Floer correction term associated to $g$  viewed as a Spin$^c$--structure.   Basic results concerning $\text{Spin}^c(M)$, the $d$--invariant, and its basic properties are in~\cite{MR1957829}.   Further details and examples are provided in~\cite{MR3259765}.

If $M$ bounds a rational homology four-ball, then there is a subgroup $\calm \subset H_1(M)$ such that: $\big|\calm\big| ^2= \big|H_1(M)\big|$;  the nonsingular linking form vanishes on $\calm$; and $d(M, g) = 0$ for all $g \in \calm$.    We can equivalently view $\calm \subset H^2(M)$.

\begin{theorem}\label{thm:main:d}  Let $p$ and $q$ be distinct odd primes and let $M$ be a three-manifold satisfying $H_1(M) \cong \Z_{p^2} \oplus \Z_{q^2}$,  generated by elements $a$ and $b$ of order $p^2$ and $q^2$, respectively.  If  $M$ is rationally homology cobordant  to a connected sum $ M_1 \cs M_2$ where $H_1(M_1)_{(q)} = 0$ and $H_1(M_2)_{(p)} = 0$, then the value of    $$d(K,ipa +jqb) - d(K,ipa) - d(K,jqb)$$ is independent of $i$ and $j$. 
\end{theorem}

\begin{proof}  Since $M$ and $M_1 \cs  M_2$  are rationally homology  concordant, $M \cs  -M_1 \cs  -M_2$ bounds a           rational homology $4$--ball $W$.  The image of $H^2(W)$   in $H\!:= H^2(M \cs -M_1 \cs -M_2)$ is the desired  subgroup $\calm$ satisfying $\big|\calm\big|^2 = \big|H\big|$.  Furthermore, $\calm$ is self-annihilating with respect to the nonsingular linking form.  The order of a self-annihilating subgroup of a group $G$ is of order at most $\sqrt{|G|}$; it follows that  the subgroup $\calm_{(p)} \subset H(M)_{(p)} \oplus H(-M_1)_{(p)}$ cannot be contained in $H(-M_1)_{(p)}$.  In particular, some element of the form $(x_p, y_p) \in  H(M)_{(p)} \oplus H(-M_1)_{(p)}$ with $x_p \ne 0 \in Z_{p^2}$ is contained in $\calm$.  By taking a multiple, we can assume $x_p = pa$.  Similarly, there is an element $(qb, y_q) \in  H(M) \oplus H(-M_2)$ in $\calm_q$.

Notice that we are viewing $H_1(-M_1) \subset H_1(-M_1) \oplus H_1(-M_2)$, so in this sense  $y_p$ can be interpreted as an ordered pair $(y_p, 0) \in H_1(-M_1) \oplus H(-M_2)$; similarly, $y_q$ represents an ordered pair   $(0,y_q) \in H_1(-M_1) \oplus H_1(-M_2)$. The correction term is additive under connected sum and vanishes for elements in $\calm$.  Thus for any $i$ and $j$:
$$d(M, ipa) - d(M_1, iy_p) - d(M_2,0)= 0,$$
$$d(M, jqb)- d(M_1,0) - d(M_2, jy_q) = 0,$$
and
$$d(M, ipa + jqb ) - d(M_1, iy_p)  - d(M_2, jy_q) = 0.$$
Subtracting the first two equations from the third shows that for all $i$ and $j$,
$$d(M, ipa + jqb ) - d(M, ipa)  - d(M, jqb) = -d(M_1, 0)   -d(M_2, 0).$$ The right hand side is independent of $i$ and $j$.

\end{proof}


\section{The knots $K_n(J_n)$  and their Alexander polynomials}\label{sec:alex}

As described in the introduction, we are considering the knots illustrated schematically in  Figure~\ref{fig:knot}.    

\begin{theorem}  Let $K_n = K_n(J_n)$. For  $n$   odd, the Alexander polynomial is given by $$\Delta_{K_n}(t) = \left( \frac{ t^{n}+1}{t+ 1}\right)^2.$$ \end{theorem}

\begin{proof}
This knot is a winding number two companion of $J_n$.  Since $J_n$ has Alexander polynomial one, a  standard formula for the Alexander polynomial of a satellite knot~\cite{MR0035436} implies  that the Alexander polynomial of $K_n(J_n)$ is the same as that for $K_n(U)$, where $U$ is the unknot.  In this case, a simple manipulation shows that  $K_n(U) = P(n, -n, n-1)$,   a three-stranded pretzel knot, illustrated in the case of $n=3$ in Figure~\ref{fig:pretzel}.  To compute its Alexander polynomial, we consider instead the Conway polynomial $\nabla_{K_n}(z)$.  (Recall that for an arbitrary oriented knot $K$, $\Delta_K(t) = \nabla_K(t^{1/2} - t^{-1/2})$.)

\begin{figure}[h]
\fig{.5}{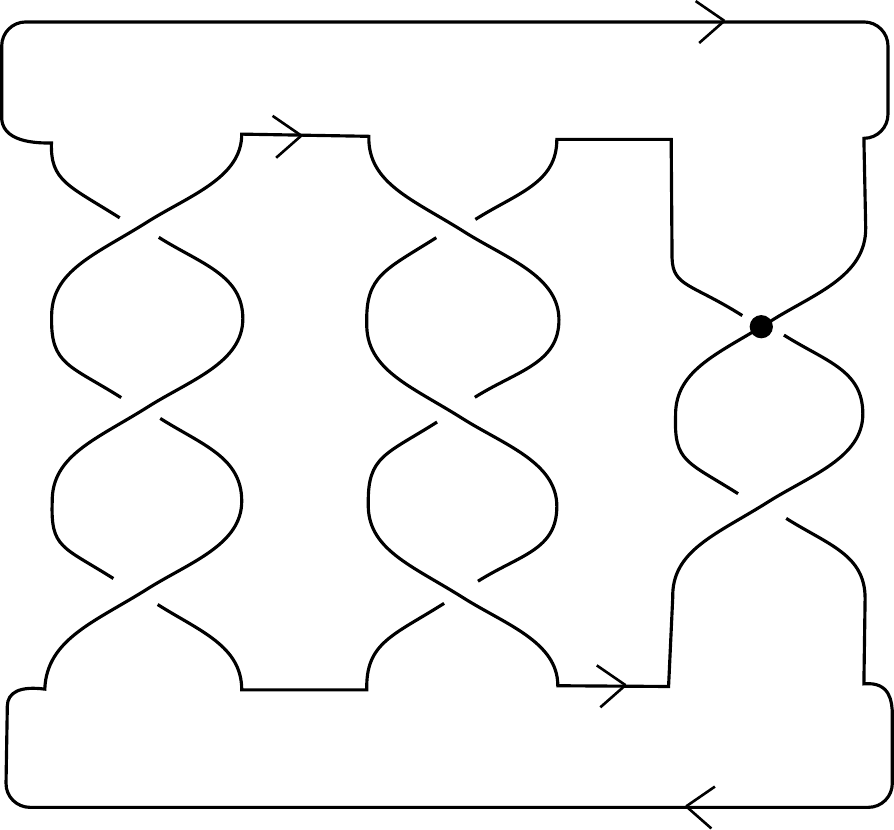}
\caption{The pretzel knot $P(3,-3,2)$}\label{fig:pretzel}
\end{figure}

The standard crossing change formula  for the Conway polynomial is $$\nabla_+(z) - \nabla_-(z) = -z \nabla_s(z),$$    where $\nabla_\pm$ denotes the Conway polynomial of  an oriented link with a specified crossing made positive or negative  and $\nabla_s$ is the Conway polynomial of the    link formed by smoothing that same crossing.  This can be applied to a crossing on the right-most band of the pretzel knot, indicated by the dot in Figure~\ref{fig:pretzel}.  Smoothing the crossing yields an unlink, which has  0 Conway polynomial.  Changing the crossing removes two half-twists.  Thus, $\nabla_{P(n, -n, n-1)}(z) =\nabla_{P(n, -n, n-3)}(z)$.  Since $n$ is odd, continuing in this way removes the right-most crossings, ultimately yielding the connected sum $-T(2,n) \cs T(2,n)$.   The Alexander polynomials of torus knots is well-known; in this case it is $$\Delta_{T(2,n)}(t) = \frac{(t^{2n} -1)(t-1)}{(t^2 - 1)(t^n -1)} = \frac{t^n +1}{t+1}.$$
\end{proof}

\begin{lemma}\label{lemma:cyclotomic} For distinct odd primes $p$ and $q$, there is the following identity, where the $\phi_i$ are cyclotomic polynomials.
$$ \frac{t^{pq} +1}{t+1} =  \phi_{2p}(t) \phi_{2q}(t) \phi_{2pq}(t).$$ Furthermore, $\phi_{2p}(-1) = p$, $\phi_{2q}(-1) = q$, and $\phi_{2pq}(-1) = 1$.
\end{lemma}

\begin{proof}  The polynomial $t^n -1$ has factors $\phi_d(t)$ for all divisors $d$ of $n$.  Thus 

\begin{equation}{t^{2pq} -1} = \phi_{2pq}(t)\phi_{2p}(t)\phi_{2q}(t) \phi_{p}(t)\phi_{q}(t)\phi_2(t)\phi_1(t)
\end{equation}
and $${t^{ pq} -1} = \phi_{ pq}(t)\phi_{ p}(t)\phi_{ q}(t)\phi_1(t).$$ Dividing the first equation by the second, and then dividing by $\phi_2(t) = t +1$ yields   $$ \frac{t^{pq} +1}{t+1} =  \phi_{2pq}(t) \phi_{2p}(t) \phi_{2q}(t) .$$  

L'Hopital's rule can be used to determine that the left hand side evaluated at $-1$ is $pq$.  Thus, if we show $\phi_{2p}(-1) = p$, and, similarly, $\phi_{q}(-1) = q$, we are done.  Proceeding as before, $$t^{2p} -1 = \phi_{2p}(t) \phi_p(t) \phi_2(t) \phi_1(t) $$ and $$t^{p} -1 =  \phi_p(t)  \phi_1(t). $$  Dividing yields   $$\frac{ {t^p +1}}{t+1}  =  \phi_{2p}(t). $$  In this case, L'Hopital's rule shows that $\phi_{2p}(-1) = p$.

\end{proof}

\begin{corollary} If $n = pq$, where $p$ and $q$ are distinct odd primes, then $\Delta_{K_n(J_n)}(t) =\big( \phi_{2p}(t)\phi_{2q}(t)\phi_{2pq}(t)\big)^2$, where the $\phi_k(t)$ are cyclotomic polynomials.  For $r$ an  odd prime,  $\phi_{2r}(-1) = r$ and for  a product of two distinct odd primes, $\phi_{2pq}(-1) = 1$.
\end{corollary}

\section{Computing $d(M(K), i)$; the completion of the proof of Theorem~\ref{thm:main}}

We will now restrict to the  case of $n = 15$.     The methods of~\cite{MR593626} apply to show that the  2--fold branched cover of $S^3$ branched over $K_n(J_n)$ is given by $15^{2}$--surgery on $T_{14,15} \cs 22D(T_{2,3})$.  We continue to denote $K_{15}(J_{15})$ by $K_{15}$  and denote its 2--fold  branched cover simply by $M$.  Note that $H_1(M) \cong \Z_{15^2}$.  

Suppose now that $K_{15}$ were concordant to a connected sum of three knots, one with Alexander polynomial $ \phi_{6}(t) ^{m_1} $, one with Alexander polynomial   $ \phi_{10}(t) ^{m_2}$, and one with Alexander polynomial $   \phi_{30}(t) ^{m_3}$.   Then, as described in the introduction, $M$ would by rationally homology cobordant to $M_1 \cs M_2$, where $H_1(M_1)_{(5)} = 0$ and $H_1(M_2)_{(3)} = 0$.  Thus,  Theorem~\ref{thm:main:d} would apply. 

In~\cite[Section 6]{MR2955197}, an algorithm is presented for computing the values of the $d$--invariants of $M(K_n(J_n))$.  Here is   the results of the computation; readers are referred to~\cite{MR2955197} for general background.  In Appendix~\ref{app:dcomp} we provide a summary of the details of this specific computation.  Since $H_1(M) \cong Z_{225}$, the order 9 subgroup is generated by $a = 25$ and the order 25 subgroup is generated by $b = 9$.   The values that result from the computation  of $d(M, ipa + jqb ) $ are as shown in Table~\ref{table:d}. 

\begin{table}[h]
\begin{center}
\begin{tabular}{ | c | c | c | c | } 
\hline
&     $i = 0  $        & $i=1$        &    $ i = 2$   \\ 
\hline
$j=0$         &          		22             &     		14       & 	14   \\ 
\hline
$j=1$         &              	18          &    		6       &  			20  \\ 
\hline
$j=2$         &           		10              &   		16      & 			22  \\ 
\hline
$j=3$         &          		10               &    	 22     & 		16   \\ 
\hline
$j=4$         &         		18                &     	20       	&	 	6   \\ 
\hline
\end{tabular}
\vskip.1in
\end{center}
\caption{Values of  $d(M, ipa + jqb )  $}\label{table:d}
\end{table}

The values of  $d(M, ipa + jqb ) - d(M, ipa)  - d(M, jqb)$ are as shown (with sign reversed for readability) in Table~\ref{table:d2};  they are not all equal.

\begin{table}[h]
\begin{center}
\begin{tabular}{ | c | c | c | c | } 
\hline
&     $i = 0  $        & $i=1$        &    $ i = 2$   \\ 
\hline
$j=0$         &          		22             &     		22      & 	22  \\ 
\hline
$j=1$         &              	22          &    		26       &  			12  \\ 
\hline
$j=2$         &           		22             &   		8      & 			2  \\ 
\hline
$j=3$         &          		22              &    	 2     & 		8  \\ 
\hline
$j=4$         &         		22               &     	12       	&	 	26   \\ 
\hline
\end{tabular}
\vskip.1in
\end{center}
\caption{Values of  $-\big(d(M, ipa + jqb ) - d(M, ipa)  - d(M, jqb)\big)$}\label{table:d2}
\end{table}

\subsection{Infinite families}  Let $\{p_i\}$ be an infinite increasing sequence of  primes for which $p_i \equiv 3 \mod 4$ if and only if $i$ is odd.  Let $\calp_0 = \cup_{i =1  }^\infty  \{\phi_{2p_i -1 }(t), \phi_{2p_i  }(t), \phi_{2 p_{2i-1} p_{2i}}(t)\}$.  Most of the previous argument is easily generalized.  The only step that we have not been able to complete in general is the computation of the $d$--invariants.   Our expectation is that this would lead to the conclusion that $$ \oplus_{p \in \calp_0} \calt^p_\Delta \to  \calt^{\calp_{0}}_\Delta$$ is not surjective.


\appendix
\section{Computation of $d$--invariants}\label{app:dcomp}

Here we describe the computation of the $d$--invariants for  the 2--fold branched cover of $S^3$  branched over $K_{15}(J_{15})$.  A related example was presented in~\cite[Section 6]{MR2955197} with further background material  but lacking a few of the details that we provide here.   

For this knot, recall that $J_{15}$ is the connected sum of 11 copies of the untwisted Whitehead double of the trefoil knot, $Wh(T(2,3))$.  The 2--fold branched  cover, which we denote by $M_{15}$, can be described as $15^2$ surgery on the knot $L = T(14,15) \cs_{22}Wh(T(2,3))$.  That is, $M_{15} = S^3_{15^2}(L)$.   

Our goal is to compute the $d$--invariants  $d(M_{15} , m)$, which in~\cite{MR2955197} were   denoted $d(M_{15}, \spinc_m)$.   The first step is to determine  the Heegaard Floer knot complex  $ \cfk(L)$.  This is a chain complex with coefficients in  $\F$,  the field with two elements.  It  is  $\Z$--graded, supports  two increasing filtrations, and is a free $\F[U, U^{-1}]$--module.  The action of $U$ lowers gradings by 2 and filtration levels by 1. 

According to~\cite{MR2065507}, complexes of connected sums of knots are   the tensor products of the corresponding complexes for the individual knots, so we need first to describe $\cfk(T(14,15))$ and $\cfk(\cs_{22}Wh(T(2,3)))$.  In~\cite{MR2955197} it is shown that these are of the form   $\left(C_1 \otimes \F[U,U^{-1}]\right) \oplus A_1$  and $\left(C_2 \otimes \F[U,U^{-1}]\right) \oplus A_2$, where $A_1$ and $A_2$ are acyclic.   The acyclic summands do not affect the value of the $d$--invariant of surgery on the knots, so can be ignored.   Both $C_1$ and $C_2$ are {\it stairway} complexes; in particular, they are freely generated by elements of grading 0   and of grading 1.  Each has one dimensional homology, and that homology is at grading 0.  All the grading 0 generators are homologous cycles.   

For the complex $C_1$, the grading 0 generators   have bifiltration levels given by the following set, along with the symmetric values; for example, since $(0,105)$ is listed, there is also a generator at bifiltration level $(105,0)$.  There are 15 generators; the following eight and their reflections: 
$$\{(0, 105), (1, 91), (3, 78), (6, 66), (10, 55), (15, 45), (21, 
36), (28, 28)\}.$$

The corresponding list of the 23 generators of  $C_2$ are given in the following list, where we present one element from each symmetric pair.

$$\{ (0, 22), (1, 21), (2, 20), (3, 19), (4, 18), (5, 17), (6, 16), (7, 
15), (8, 14), (9, 13), (10, 12), (11, 11)\}.$$

The tensor product of the two complexes has $15 \times 23 = 345$ generators of grading 0, all of which are cycles representing the generator of homology.  The set of all bifiltration levels of these generators is formed by taking all possible sums of the bifiltration levels from each set.    Call the set of these bifiltration levels $\cals$.

In~\cite[Theorem 5.3, Section 6]{MR2955197}, it is described how    the value of $d$--invariant $d(M, m)$ is computed using these generators.  Here is a concise summary.   For any $m$ satisfying $|m| \le 112$, for each generator at filtration level $(\alpha, \beta)$, one computes the value of the function  $\Psi(\alpha, \beta)$ defined by 
$$
\Psi(\alpha , \beta) =
\begin{cases}
\beta -m, &\text{if\ \ } \beta -\alpha \ge m,\\
\alpha, &\text{if\ \ } \beta-\alpha < m.\\
\end{cases}
$$
Next, one lets $$\delta_m(\cals) = \min\{\Psi(s)\ \big| \  s \in \cals\}.$$   The next result presents the final result that is needed to complete the computation.

\begin{theorem}  For   $m$ satisfying $ |m| \le 112$,  $$d(M, m) = -2\delta(m) - \frac{-(2m - 225)^2 + 225}{(4)(225)}.$$
\end{theorem} 

With these results, the computer computations of the values in Table~\ref{table:d} are straightforward.




\begin{bibdiv}
\begin{biblist}

\bib{MR593626}{article}{
      author={Akbulut, Selman},
      author={Kirby, Robion},
       title={Branched covers of surfaces in {$4$}-manifolds},
        date={1979/80},
        ISSN={0025-5831},
     journal={Math. Ann.},
      volume={252},
      number={2},
       pages={111\ndash 131},
         url={https://mathscinet.ams.org/mathscinet-getitem?mr=593626},
      review={\MR{593626}},
}

\bib{MR900252}{incollection}{
      author={Casson, A.~J.},
      author={Gordon, C.~McA.},
       title={Cobordism of classical knots},
        date={1986},
   booktitle={\`a la recherche de la topologie perdue},
      series={Progr. Math.},
      volume={62},
   publisher={Birkh\"auser Boston, Boston, MA},
       pages={181\ndash 199},
        note={With an appendix by P. M. Gilmer},
      review={\MR{900252}},
}

\bib{arxiv:1910.14629}{article}{
      author={Cha, Jae~Choon},
       title={Primary decomposition in the smooth concordance group of
  topologically slice knots},
        date={2019},
     journal={arxiv.org},
        ISSN={abs/math/},
      eprint={arxiv.org/abs/math/1910.14629},
         url={https://arxiv.org/abs/math/1910.14629},
}

\bib{MR976591}{article}{
      author={Cochran, Tim~D.},
      author={Gompf, Robert~E.},
       title={Applications of {D}onaldson's theorems to classical knot
  concordance, homology {$3$}-spheres and property {$P$}},
        date={1988},
        ISSN={0040-9383},
     journal={Topology},
      volume={27},
      number={4},
       pages={495\ndash 512},
         url={https://mathscinet.ams.org/mathscinet-getitem?mr=976591},
      review={\MR{976591}},
}

\bib{MR1334309}{article}{
      author={Endo, Hisaaki},
       title={Linear independence of topologically slice knots in the smooth
  cobordism group},
        date={1995},
        ISSN={0166-8641},
     journal={Topology Appl.},
      volume={63},
      number={3},
       pages={257\ndash 262},
         url={https://mathscinet.ams.org/mathscinet-getitem?mr=1334309},
      review={\MR{1334309}},
}

\bib{MR0211392}{article}{
      author={Fox, Ralph~H.},
      author={Milnor, John~W.},
       title={Singularities of {$2$}-spheres in {$4$}-space and cobordism of
  knots},
        date={1966},
        ISSN={0030-6126},
     journal={Osaka J. Math.},
      volume={3},
       pages={257\ndash 267},
         url={https://mathscinet.ams.org/mathscinet-getitem?mr=0211392},
      review={\MR{0211392}},
}

\bib{MR1201584}{book}{
      author={Freedman, Michael~H.},
      author={Quinn, Frank},
       title={Topology of 4-manifolds},
      series={Princeton Mathematical Series},
   publisher={Princeton University Press, Princeton, NJ},
        date={1990},
      volume={39},
        ISBN={0-691-08577-3},
         url={https://mathscinet.ams.org/mathscinet-getitem?mr=1201584},
      review={\MR{1201584}},
}

\bib{MR679066}{article}{
      author={Freedman, Michael~Hartley},
       title={The topology of four-dimensional manifolds},
        date={1982},
        ISSN={0022-040X},
     journal={J. Differential Geom.},
      volume={17},
      number={3},
       pages={357\ndash 453},
         url={https://mathscinet.ams.org/mathscinet-getitem?mr=679066},
      review={\MR{679066}},
}

\bib{MR3466802}{article}{
      author={Hedden, Matthew},
      author={Kim, Se-Goo},
      author={Livingston, Charles},
       title={Topologically slice knots of smooth concordance order two},
        date={2016},
        ISSN={0022-040X},
     journal={J. Differential Geom.},
      volume={102},
      number={3},
       pages={353\ndash 393},
         url={https://mathscinet.ams.org/mathscinet-getitem?mr=3466802},
      review={\MR{3466802}},
}

\bib{MR2955197}{article}{
      author={Hedden, Matthew},
      author={Livingston, Charles},
      author={Ruberman, Daniel},
       title={Topologically slice knots with nontrivial {A}lexander
  polynomial},
        date={2012},
        ISSN={0001-8708},
     journal={Adv. Math.},
      volume={231},
      number={2},
       pages={913\ndash 939},
         url={https://mathscinet.ams.org/mathscinet-getitem?mr=2955197},
      review={\MR{2955197}},
}

\bib{MR3259765}{article}{
      author={Kim, Se-Goo},
      author={Livingston, Charles},
       title={Nonsplittability of the rational homology cobordism group of
  3-manifolds},
        date={2014},
        ISSN={0030-8730},
     journal={Pacific J. Math.},
      volume={271},
      number={1},
       pages={183\ndash 211},
         url={https://mathscinet.ams.org/mathscinet-getitem?mr=3259765},
      review={\MR{3259765}},
}

\bib{MR0253348}{article}{
      author={Levine, J.},
       title={Invariants of knot cobordism},
        date={1969},
        ISSN={0020-9910},
     journal={Invent. Math. 8 (1969), 98--110; addendum, ibid.},
      volume={8},
       pages={355},
         url={https://doi.org/10.1007/BF01404613},
      review={\MR{0253348}},
}

\bib{MR1957829}{article}{
      author={Ozsv{\'a}th, Peter},
      author={Szab{\'o}, Zolt{\'a}n},
       title={Absolutely graded {F}loer homologies and intersection forms for
  four-manifolds with boundary},
        date={2003},
        ISSN={0001-8708},
     journal={Adv. Math.},
      volume={173},
      number={2},
       pages={179\ndash 261},
         url={https://doi.org/10.1016/S0001-8708(02)00030-0},
      review={\MR{1957829}},
}

\bib{MR2065507}{article}{
      author={Ozsv\'{a}th, Peter},
      author={Szab\'{o}, Zolt\'{a}n},
       title={Holomorphic disks and knot invariants},
        date={2004},
        ISSN={0001-8708},
     journal={Adv. Math.},
      volume={186},
      number={1},
       pages={58\ndash 116},
         url={https://mathscinet.ams.org/mathscinet-getitem?mr=2065507},
      review={\MR{2065507}},
}

\bib{MR0035436}{article}{
      author={Seifert, H.},
       title={On the homology invariants of knots},
        date={1950},
        ISSN={0033-5606},
     journal={Quart. J. Math., Oxford Ser. (2)},
      volume={1},
       pages={23\ndash 32},
         url={https://mathscinet.ams.org/mathscinet-getitem?mr=0035436},
      review={\MR{0035436}},
}

\bib{MR0467764}{article}{
      author={Stoltzfus, Neal~W.},
       title={Unraveling the integral knot concordance group},
        date={1977},
        ISSN={0065-9266},
     journal={Mem. Amer. Math. Soc.},
      volume={12},
      number={192},
       pages={iv+91},
         url={https://mathscinet.ams.org/mathscinet-getitem?mr=0467764},
      review={\MR{0467764}},
}

\bib{MR0117736}{article}{
      author={Terasaka, Hidetaka},
       title={On null-equivalent knots},
        date={1959},
     journal={Osaka Math. J.},
      volume={11},
       pages={95\ndash 113},
         url={https://mathscinet.ams.org/mathscinet-getitem?mr=0117736},
      review={\MR{0117736}},
}

\end{biblist}
\end{bibdiv}

\end{document}